\documentclass{article} 

\usepackage{tikz} 
 
\usepackage[english]{babel} 
\usepackage[utf8]{inputenc} 
\usepackage[T1]{fontenc} 
\usepackage[a4paper,left=3cm,right=3cm,top=3cm,bottom=3.2cm]{geometry} 
\usepackage{tikz} 
\usetikzlibrary[topaths]

\usepackage{threeparttable}

\usepackage{amsmath, mathrsfs, amssymb, amsthm, mathtools, bbm, bm} 

\usepackage{subfig} 

\usepackage{caption,float} 

\usepackage{makecell} 
\usepackage{enumitem} 

\usepackage[hidelinks]{hyperref} 
\usepackage{cleveref}

\newtheorem{theorem}{Theorem}
 
\newtheorem{definition}{Definition}
\newtheorem{lemma}{Lemma}
\newtheorem{remark}{Remark}

\newtheorem{proposition}{Proposition}
\newtheorem{assumption}{Assumption}

\newtheorem*{theorem*}{Theorem}
\newtheorem*{example*}{Example} 
\newtheorem*{definition*}{Definition}
\newtheorem*{lemma*}{Lemma}
\newtheorem*{remark*}{Remark}
\newtheorem*{corollary*}{Corollary}
\newtheorem*{proposition*}{Proposition}
\newtheorem*{assumption*}{Assumption}
\newtheorem*{claim*}{Claim}

\newtheoremstyle{TheoremNum}
        {\topsep}{\topsep}              
        {\itshape}                      
        {}                              
        {\bfseries}                     
        {.}                             
        { }                             
        {\thmname{#1}\thmnote{ \bfseries #3}}
\theoremstyle{TheoremNum}

\newtheoremstyle{LemmaNum}
        {\topsep}{\topsep}              
        {\itshape}                      
        {}                              
        {\bfseries}                     
        {.}                             
        { }                             
        {\thmname{#1}\thmnote{ \bfseries #3}}
\theoremstyle{LemmaNum}

\usepackage{hyperref} 

\usepackage[numbers]{natbib}
\usepackage{float}

\usepackage{amsmath, mathrsfs, amssymb, amsthm, mathtools, bbm, bm} 

\usepackage{makecell} 

\usepackage{subfig}
\usepackage{float}

\makeatletter
\newcommand{\newparallel}{\mathrel{\mathpalette\new@parallel\relax}}
\newcommand{\new@parallel}[2]{%
  \begingroup
  \sbox\z@{$#1T$}
  \resizebox{!}{\ht\z@}{\raisebox{\depth}{$\m@th#1/\mkern-5mu/$}}%
  \endgroup
}
\makeatother

\usepackage{enumitem} 

\usepackage{algorithm} 
\usepackage{algcompatible}

\renewcommand{\Pr}{ \mathbb{P} }

\newcommand{\R}{\mathbb{R}}

\newcommand{\E}{\mathbb{E}}

\renewcommand{\[}{\left[ }
\renewcommand{\]}{\right] }

\newcommand{\<}{\left< }
\renewcommand{\>}{\right> }

\renewcommand{\(}{\left( }
\renewcommand{\)}{\right) }

\def\highlight{0}

\begin{document} 

\title{Revisiting Stochastic Gradient Descent for Strongly Convex Objectives: Tight Uniform-in-Time Bounds}  

\author{Kang Chen\footnote{The authors are listed in alphabetical order.} \quad Yasong Feng${}^*$ \quad Tianyu Wang${}^*$ } 

\date{} 

\maketitle 

\begin{abstract} 
Stochastic optimization via Stochastic Gradient Descent (SGD) is a fundamental problem in statistics and optimization. 
This paper revisits Stochastic Gradient Descent (SGD) for strongly convex objectives, establishing tight, uniform-in-time convergence bounds. We prove that, with probability at least $1 - \beta$, a convergence rate of order $\frac{\log \log k + \log (1/\beta)}{k}$ simultaneously holds for all $ k \in \mathbb{N}_+ $, and demonstrate this bound is tight up to constant factors. We also provide an improved last-iterate convergence rate for such objectives. 
While focused on strongly convex objectives, our results generalize to the Polyak--\L ojasiewicz functions and indicate an $\mathcal{O}(k^{-1} \log \log k)$ convergence rate for contractive stochastic approximation with additive noise.
\end{abstract} 
 
\section{Introduction}

Stochastic optimization for smooth, strongly convex objectives is a foundational yet enduring challenge in statistics and optimization, with roots tracing back to the pioneering work of \citet{10.1214/aoms/1177729586} and \citet{10.1214/aoms/1177729392}. The core problem is the minimization of a function $f$,
\begin{align*}
\min_{x \in \R^d} f (x),
\end{align*}
where the optimization algorithm must proceed using only noisy estimates of the gradient of $f$. 
Over the decades, this problem has been extensively studied \citep[e.g.,][ and references therein]{lan2020first,liu2024revisiting,2506.23335,doi:10.1287/stsy.2022.0097, 10.1214/24-AAP2143}. 

\subsection{Uniform-in-Time Bounds}

Uniform-in-time bounds are a cornerstone of sequential statistical analysis. These bounds provide simultaneous guarantees that hold with high probability for all time steps $k\in\mathbb{N}$. Their \emph{uniform} validity across the whole time horizon makes them indispensable for sequential decision-making applications such as adaptive sampling, real-time inference, and stopping-time analyses.

In the context of parameter estimation with i.i.d. samples, such time-uniform high probability bounds are often referred to as confidence sequences \citep{darling1967confidence, lai1976confidence}. Due to their time-uniform nature, confidence sequences have become widely utilized in sequential decision-making tasks \citep{jennison1989interim, lai1984incorporating, jamieson2014lil, kaufmann2016complexity, howard2021time, johari2022always, waudby2024estimating}. 
{Unlike the typical $\Theta(k^{-1/2})$ shrinkage rate of confidence intervals, confidence sequences exhibit a slower rate of $\mathcal{O}(k^{-1/2} \sqrt{\log\log k})$.}
A recent work of \citet{duchi2024information} has further established that the $\sqrt{ \log\log k}$ term is not merely an artifact but a necessary component of such bounds. 

Uniform-in-time guarantees are particularly important in stochastic optimization, since such guarantees implies convergence in stopping times.\footnote{Convergence in stopping time and uniform-in-time convergence are equivalent; See Lemma 3 in \cite{2506.23335}.} In practice, stochastic optimization algorithms are usually terminated based on dynamic, data-driven criteria. A common example is early stopping in large-scale machine learning, where training is halted once validation performance plateaus or begins to degrade—a widely adopted heuristic to prevent overfitting \citep[e.g.,][]{prechelt2002early,dodge2020fine}. From a mathematical standpoint, such termination rules define stopping times. To illustrate, consider a simple stopping criterion for an iterative stochastic algorithm:  
\begin{align}  
    \tau := \min \{ k \in \mathbb{N} : \| x_k - x_{k-1} \| \le \varepsilon \}, \label{eq:toy}  
\end{align}  
where $ x_k $ denotes the iterate at step $ k $ and $ \varepsilon $ is a predefined tolerance threshold. Importantly, $ \tau $ is inherently random. 

{\if\highlight1\color{red}\fi
The problem of obtaining uniform-in-time concentration bounds has also been explored in the context of contractive stochastic approximation algorithms \citep{doi:10.1287/stsy.2022.0097, 10.1214/24-AAP2143}. In particular, \citet{10.1214/24-AAP2143} obtained an $\mathcal{O} (k^{-1}\log k)$ uniform-in-time bound for contractive stochastic approximation with additive noise. In this work, 
we provide a refined analysis of stochastic gradient descent under strongly convex settings that yields an $\mathcal{O}(k^{-1} \log \log k)$ uniform-in-time convergence rate -- an improvement over the previously known $\mathcal{O}(k^{-1} \log k)$ rate -- and show that this rate cannot be improved in general. 
}

\subsection{Our Results}  

In this work, we establish a uniform-in-time convergence bound for the standard stochastic gradient descent (SGD) on strongly convex functions: 
the standard SGD algorithm attains, 
with probability greater than $1 - \beta$, a \textit{uniform} convergence rate of $\frac{\log \log k + \log(1/\beta)}{k}$ across all iterations $k \in \mathbb{N}_+$. Specifically, we prove the following upper bound: 

\begin{itemize}
    \item For any {\if\highlight1\color{red}\fi$\mu$-strongly convex} and $L$-smooth objective $f$, the standard stochastic gradient descent (SGD) algorithm with standard step size $ \eta_k = \Theta \( \frac{1}{k} \) $ for all $k \in \mathbb{N}_+$ satisfies: There exists constants $C_1$ and $C_2$ (See Theorem \ref{thm:upper} for details on $C_1$ and $C_2$), such that for any $\beta \in (0,1)$, 
    \begin{align*} 
        &\; \Pr \( f (x_k) - f^* \le C_1 \cdot \frac{\log \log k + \log (1/\beta) + C_2 }{k} , \; \forall k \in \mathbb{N}_+ \) \\
        \ge&\;  1 - \beta , 
    \end{align*} 
    where $f^*$ is the minimum of $f$, and $x_k$ is governed by the standard SGD algorithm. 
\end{itemize}

{\if\highlight1\color{red}\fi To achieve this rate, we construct a super-martingale for the conditional moment generating function and apply Ville's maximal inequality sequentially over geometrically expanding intervals of the form $[2^k,2^{k+1})$. A final union bound across this decomposition yields the uniform-in-time guarantee of order $\mathcal{O}(k^{-1} \log \log k)$. This technique improves upon prior results of order $\mathcal{O}(k^{-1}\log k)$, which rely on either a direct application of Ville's inequality or a uniform bound over all non-negative integers $\mathbb{N}$ with no partitioning. } 





In addition, we prove that the above upper bound is tight. More precisely, via an argument that reduces uniform-in-time optimization to uniform-in-time testing, we derive the following lower bound for \emph{uniform-in-time convergence rate}:
\begin{itemize}
    \item 
    Let $\beta \in (0,1) $. If $\varepsilon (k)$ is an error sequence such that for any $\mu$-strongly convex $f$, $x_k$ governed by a stochastic first-order method, it holds that 
    \begin{align*} 
        \Pr \( f (x_k) - f^* \le \varepsilon(k) , \; \forall k \in \mathbb{N}_+ \) \ge 1 - \beta, 
    \end{align*}
    then for infinitely many $k\in\mathbb{N}_+$, we have 
    \begin{align*} 
        \varepsilon(k)\geq\frac{\sigma^2(1-\beta)}{48\mu}\cdot\frac{\log\log k}{k}, 
    \end{align*} 
    where $\sigma^2$ is the noise variance. See Theorem \ref{thm:lower} for details. 
\end{itemize}

{\if\highlight1\color{red}\fi
\begin{remark}
    Although our primary focus is strongly convex objectives, our results naturally extend to the Polyak-Łojasiewicz (PL) functions. Moreover, our analysis implies an $\mathcal{O}(k^{-1} \log \log k)$ convergence rate for contractive stochastic approximation with additive noise, as discussed in \ref{appendix:a}. 
\end{remark}
}

Also, as an intermediate result, we succinctly prove an improved last-iterate convergence rate of stochastic gradient descent on strongly convex objectives. Compared to the state-of-the-art results by \citet{liu2024revisiting}, our last-iterate bound removes the logarithmic factors in $k$. A summary of comparison with state-of-the-art results for SGD analysis on strongly convex objectives is presented in Table \ref{tab:compare}. 

\begin{table*}[h!]
    \centering
    \begin{threeparttable}
        \caption{High probability convergence rates for standard stochastic gradient descent on strongly convex objectives. The second column contains the error bounds with probability at least $1-\beta$. 
        } 
        \label{tab:compare}
        \begin{tabular}{|c|c|c|c|}        
            \hline 
             & \makecell{High-probability \\ convergence order, \\ with probability $\ge 1-\beta$} & Type & Domain \\
            \hline 
           \makecell{Theorem 3.5 in \cite{liu2024revisiting},\\
           the case for unknown\\total iterate count.\\Existing state-of-the-art.}
            & $\mathcal{O}\left(\frac{  \log k \cdot { \log (1/\beta) } }{k}\right)$ 
            & {\makecell{Non-random\\last iterate $k$.~\tnote{1}}} 
            & $\mathbb{R}^d$ 
            \\ \hline
            
            \textbf{\makecell{This work\\(Lemma \ref{lemma5})}} 
            & $ \mathcal{O}\left(\frac{\log (1/\beta)}{k}\right) $ 
            & \makecell{Non-random\\last iterate $k$.}  
            & $\mathbb{R}^d$ 
            \\ \hline
            
            \textbf{\makecell{This work\\(Theorem \ref{thm:upper})}} 
            & $ \mathcal{O}\left(\frac{ \max\{ 1, \log \log k \} + \log (1/\beta)}{k}\right) $
            & \makecell{Uniform-in-time.\\Simultaneously holds\\for all $k \in \mathbb{N}_+$.} 
            & $\mathbb{R}^d$ \\ \hline
            
            \textbf{\makecell{This work\\(Theorem \ref{thm:lower})}} 
            & $ \Omega\left(\frac{ (1-\beta)\log\log k}{k}\right) $
            & \makecell{Infinitely often.\\Holds for infinitely\\many $k \in \mathbb{N}_+$.} 
            & $\mathbb{R}^d$ 
            
            \\ \hline
        \end{tabular}
        \begin{tablenotes}
            \item[1] In this context, `Non-random last iterate' refers to the setting where the algorithm's total iteration count is unknown before the algorithm starts, and the algorithm's output is simply the last iterate. 
            See Lemma \ref{lemma5} for details. 
        \end{tablenotes}
    \end{threeparttable}
\end{table*}

    

\section{The Uniform-in-Time Upper Bound}

Starting with $x_0 \in \mathbb{R}^d$, the standard Stochastic Gradient Descent (SGD) algorithm is governed by: 
\begin{align}
    x_{k+1} = x_k - \eta_k g(x_k; \xi_k), \quad \forall k \in \mathbb{N}, 
    \label{eq:SGD}
\end{align}
where $\eta_k \asymp \frac{1}{k} $ is the stepsize, and $g(x_k; \xi_k)$ is the stochastic gradient at $x_k$. 

To describe the stochastic process, we write 
\begin{align*}
    \mathcal{F}_k = \sigma (\xi_0, \xi_1, \cdots, \xi_{k - 1}), \quad k = 1, 2, \cdots, 
\end{align*}
for the $\sigma$-algebra generated by all randomness right before obtaining the stochastic gradient $g(x_k; \xi_k)$. We state some conventions and assumptions that will be used throughout the analysis. 

{\if\highlight1\color{red}\fi
\begin{assumption} \label{assump:f}
    The objective function $f$ is $\mu$-strongly ($\mu > 0$) convex and $L$-smooth ($0 < L < \infty$), 
    where the $L$-smoothness of $f$ means that $f$ is differentiable and for any $x, y \in \R^d$, 
    \begin{align*}
        \| \nabla f(x) - \nabla f(y) \| \leq L\|x-y\|. 
    \end{align*}
\end{assumption}
} 
{\if\highlight1\color{red}\fi
\begin{remark}
    Assumption \ref{assump:f} implies that $f$ attains its minimum $f^*>-\infty$ at some $x^*\in\R^d$ and  satisfies $\mu$-PL condition, which means that for any $x \in \R^d$, 
        $f(x) - f^* \le \frac{1}{2\mu}\|\nabla f(x)\|^2. $
\end{remark}
}


\begin{assumption} \label{assump:subGaussian}
    For all $k \in \mathbb{N}$, the stochastic gradient satisfies: 
    \begin{itemize}
        \item (Conditional unbiasedness) $\E \left[ g(x_k; \xi_k) \middle| \mathcal{F}_k \right] = \nabla f(x_k)$; 
        \item ({\if\highlight1\color{red}\fi Uniform conditional 1-sub-Gaussianity}, Theorem 2.6 in \cite{wainwright2019high})
        \begin{align*}
            \E \left[ \exp \( t \| g(x_k; \xi_k) - \nabla f(x_k) \|^2 \) \middle| \mathcal{F}_k \right] \le \frac{1}{\sqrt{1 - 2t}}, \quad \forall t \in \( 0, \frac12 \){\if\highlight1\color{red}\fi, \forall k \in \mathbb{N}}. 
        \end{align*}
    \end{itemize}
\end{assumption}

\begin{proposition}
    By Theorem 2.6 in \cite{wainwright2019high}, up to constants, the Conditional 1-sub-Gaussian condition in Assumption \ref{assump:subGaussian} implies 
    \begin{align*}
        \E \left[ \exp ( \langle \varphi, g(x_k; \xi_k) - \nabla f(x_k) \rangle ) \middle| \mathcal{F}_k \right] \le \exp \( \frac12 \| \varphi \|^2 \) , 
    \end{align*}
    for any $\varphi $ that is $\mathcal{F}_k$-measurable. 
\end{proposition}

\begin{remark}
    Without loss of generality, we assume the stochastic gradient is 1-sub-Gaussian. 
    The analysis extends straightforwardly to problems with different sub-Gaussian parameters. 
\end{remark}

For simplicity, we define 
\begin{align*}
    \Delta_k = f(x_k) - f^*, 
\end{align*}
where $\{ x_k \}_k$ is the sequence generated by SGD \eqref{eq:SGD}. 

Having established the above conventions, we now turn to the uniform-in-time analysis of the SGD algorithm {\if\highlight1\color{red}\fi\eqref{eq:SGD}}.
We start the proof with the following lemma, which establish a recursion of the conditional moment generating function of $\Delta_k$. 

\begin{lemma} \label{lemma3}
    Let Assumptions \ref{assump:f} and \ref{assump:subGaussian} hold. 
    If the stepsize sequence $\{ \eta_k \}_k$ and the nonnegative constant $t$ {\if\highlight1\color{red}\fi satisfy $\eta_k \in \left(0, \frac1L\right]$ and $t  \eta_k \in \left(0, \frac14\right]$ for all $k \in \mathbb{N}$, then the following holds, }
    \begin{align*}
        \E \left[ \exp \( t \Delta_{k+1} \) \middle| \mathcal{F}_k \right] \le \exp \( \( 1 - \frac{\mu \eta_k}{2} \) t \Delta_k + t L \eta_k^2 \), \quad \forall k \in \mathbb{N}. 
    \end{align*}
\end{lemma}

\begin{proof}
    By $L$-smoothness, we get 
\begin{align}
\begin{aligned}
    \Delta_{k+1} =&\; f(x_{k+1})-f^*  \\
    \le&\; f(x_k) + \langle \nabla f(x_k), x_{k+1} - x_k \rangle + \frac{L}{2} \| x_{k+1} - x_k \|^2 - f^* \\
    =&\; f(x_k) + \langle \nabla f(x_k), -\eta_k g(x_k; \xi_k) \rangle + \frac{L \eta_k^2}{2} \| g(x_k; \xi_k) \|^2 - f^* \\
    =&\; f(x_k) + \eta_k \langle \nabla f(x_k), \theta_k \rangle -\eta_k \| \nabla f(x_k) \|^2 + \frac{L \eta_k^2}{2} \| \theta_k \|^2  \\
    & - L \eta_k^2 \langle \theta_k, \nabla f(x_k) \rangle + \frac{L \eta_k^2}{2} \| \nabla f(x_k) \|^2 - f^* \\
    \le&\; \Delta_k + \eta_k \langle \nabla f(x_k), \theta_k \rangle - \frac{\eta_k}{2}\| \nabla f(x_k) \|^2  +\frac{L \eta_k^2}{2} \| \theta_k \|^2 - L \eta_k^2 \langle \theta_k, \nabla f(x_k) \rangle, 
\end{aligned}
\label{eq:recursion}
\end{align} 
where $\theta_k := \nabla f(x_k) - g(x_k; \xi_k)$, and the last inequality uses $\eta_k \in \(0, \frac{1}{L}\]$. 

Next consider the process $\{\exp (t\Delta_k)\}_k$. For this process we have
\begin{align}
\begin{aligned}
    & \; \E \left[ \exp \( t \Delta_{k+1} \) \middle| \mathcal{F}_k \right] \\
    \le & \; \exp \( t \Delta_k - \frac{t \eta_k}{2} \| \nabla f(x_k) \|^2\) \\
    & \; \cdot \underbrace{\E \left[ \exp \( t \( \( \eta_k - L \eta_k \) \langle \nabla f(x_k), \theta_k \rangle + \frac{L \eta_k^2}{2} \| \theta_k \|^2 \) \) \middle| \mathcal{F}_k \right] }_{=:\text{\textcircled{1}}}. 
\end{aligned}
\label{eq4}
\end{align}

By 1-sub-Gaussianity, we have
\begin{align*}
    \text{\textcircled{1}} 
    & \overset{\text{(i)}}{\le} \( \E \left[ \exp \( 2t \( \eta_k \langle \nabla f(x_k), \theta_k \rangle - L \eta_k^2 \langle \theta_k, \nabla f(x_k) \rangle \) \) \middle| \mathcal{F}_k \right] \cdot \E \left[ \exp \( L t \eta_k^2 \| \theta_k \|^2 \) \middle| \mathcal{F}_k \right] \)^\frac12 \\
    & \overset{\text{(ii)}}{\le} \( \exp \( \( 2t^2 \( \eta_k - L \eta_k^2 \)^2 \) \| \nabla f(x_k) \|^2 \) \cdot \frac{1}{\sqrt{1 - 2 L t \eta_k^2}} \) ^\frac12, 
\end{align*}
where (i) uses Cauchy-Schwarz inequality and (ii) uses sub-Gaussianity and {\if\highlight1\color{red}\fi the fact that $Lt\eta_k^2\in \(0, \frac14\] \subset \(0, \frac12\)$}. We continue the calculation for \textcircled{1} to obtain
\begin{align}
\begin{aligned}
    \text{\textcircled{1}} & \overset{\text{(i)}}{\le} \( \exp \( 2 t^2 \eta_k^2 \| \nabla f(x_k) \|^2 \) \cdot \frac{1}{\sqrt{1 - 2 L t \eta_k^2}} \)^\frac12 \\
    & \overset{\text{(ii)}}{\le} \exp \( t^2 \eta_k^2 \| \nabla f(x_k) \|^2 \) \cdot \exp \( L t \eta_k^2 \),
\end{aligned} 
\label{eq5}
\end{align} 
where (i) uses $ \left( \eta_k - L \eta_k^2 \right)^2 \le \eta_k^2 $ and (ii) uses $ L t \eta_k^2 \in \(0, \frac14\]$. 

From the definition of $t$ and $\eta_k$, we have 
\begin{align}
\begin{aligned}
    & \; t \Delta_k + \( t^2 \eta_k^2 - \frac{t \eta_k}{2} \) \| \nabla f(x_k) \|^2 \\
    \overset{\text{(i)}}{\le} & \; t \Delta_k - \frac{t \eta_k}{4} \| \nabla f(x_k)\|^2 \\
    \overset{\text{(ii)}}{\le} & \; t \Delta_k - \frac{t \mu \eta_k}{2} \Delta_k 
    =  t \( 1 - \frac{\mu \eta_k}{2} \) \Delta_k, 
\end{aligned} 
\label{eq6}
\end{align}
where (i) uses $t \eta_k \in \left(0, \frac14\right]$, and (ii) uses {\if\highlight1\color{red}\fi$\mu$-PL condition, which is implied by the $\mu$-strong convexity of $f$}. 

Collecting terms from \eqref{eq4}, \eqref{eq5} and \eqref{eq6} gives
\begin{align*} 
\E \left[ \exp \( t \Delta_{k+1} \) \middle| \mathcal{F}_k \right] \le \exp \( \( 1 - \frac{\mu \eta_k}{2} \) t \Delta_k + t L \eta_k^2 \), \quad \forall k \in \mathbb{N}, 
\end{align*}
where $\eta_k \in \left(0, \frac1L\right]$ and $t \eta_k \in \left(0, \frac14\right]$. 

\end{proof} 

Next in Lemma \ref{lemma4}, we present a version of Lemma \ref{lemma3} with a specific choice of $ t $. 

\begin{lemma} \label{lemma4}
    Let Assumptions \ref{assump:f} and \ref{assump:subGaussian} hold. If the stepsize sequence $\{ \eta_k \}_k$ and the nonnegative constant sequence $\{t_k\}_k$ {\if\highlight1\color{red}\fi satisfy $\eta_k \in \left(0, \frac1L\right]$, $t_{k + 1}  \eta_k \in \left(0, \frac14\right]$ 
    and $\(1-\frac{\mu\eta_k}{2}\)t_{k+1} \leq t_k$ for all $k\in\mathbb{N}$, then the following holds, } 
    \begin{align*}
        \E \left[ \exp \( t_{k + 1} \Delta_{k + 1} \) \middle| \mathcal{F}_k \right] \le \exp \( t_k \Delta_k + t_{k + 1} L \eta_k^2 \), \quad \forall k \in \mathbb{N}. 
    \end{align*}
\end{lemma}

\begin{proof}
    {\if\highlight1\color{red}\fi
    Note that for any $k\in\mathbb{N}$, $\eta_k$ and $t_{k+1}$ satisfy the conditions of $\eta_k$ and $t$ in Lemma \ref{lemma3}. Since $\(1-\frac{\mu\eta_k}{2}\)t_{k+1} \leq t_k$, by Lemma \ref{lemma3} we obtain that for each $k\in\mathbb{N}$,  
    \begin{align*}
        & \; \E \[ \exp\( t_{k+1}\Delta_{k+1} \) \middle| \mathcal{F}_k \] \\
        \leq & \; \exp\( t_k\Delta_k \) \cdot \exp \left\{ \( \(1-\frac{\mu\eta_k}{2}\)t_{k+1}-t_k \)\Delta_k + t_{k+1}L\eta_k^2 \right\} \\
        \leq & \; \exp\( t_k\Delta_k + t_{k+1}L\eta_k^2 \), 
    \end{align*}
    which concludes the proof. 
    }
\end{proof}


With the above two lemmas in place, we next prove an $\mathcal{O} \( \frac{\log (1/\beta)}{k} \)$ last-iterate convergence rate, 
which improves the state-of-the-art result \citep{liu2024revisiting} by a log-factor in $k$. 

\begin{lemma} \label{lemma5}
    Let Assumptions \ref{assump:f} and \ref{assump:subGaussian} hold. 
    For $k \in \mathbb{N}$, we set 
    \begin{align*}
        \eta_k = \frac{4}{\mu (k + B)}, \quad t_{k} = \frac{\mu (k + B - 1)}{16}, 
    \end{align*}
    where $B := \max \left\{ \frac{4 L}{\mu}, 3 \right\}$. Then it holds that for any $k \in \mathbb{N}, $ and any $  \beta \in (0, 1)$, 
    \begin{align*}
        \Pr \( f(x_k) - f^* \le \frac{16 \mu \log (1 / \beta) + \mu^2 (B - 1) \Delta_0 + 16 L}{\mu^2 (k + B - 1)} \) \ge 1 - \beta . 
    \end{align*}
\end{lemma}

\begin{proof}
    {\if\highlight1\color{red}\fi
    To apply Lemma \ref{lemma3} iteratively, we first verify that the chosen parameters satisfy the needed conditions. By the definition of $\eta_k$ and $t_{k+1}$, we have 
    \begin{align*}
        0 < \eta_k \leq \frac{4}{\mu B} \leq \frac{1}{L}, \quad \forall k \in \mathbb{N}. 
    \end{align*}
    And notice that 
        $ \prod_{i=0}^k \( 1 - \frac{\mu \eta_i}{2} \) = \prod_{i=0}^k \( 1 - \frac{2}{i + B} \) 
        = \prod_{i=0}^k \frac{i + B - 2}{i + B} = \frac{(B - 2) (B - 1)}{(k + B - 1) (k + B)}.  $
    Therefore, for any $0 \le j \le k$, 
    \begin{align*}
        \prod_{i=j+1}^k \( 1 - \frac{\mu \eta_i}{2} \) = \frac{\prod_{i=0}^k \( 1 - \frac{\mu \eta_i}{2} \)}{\prod_{i=0}^j \( 1 - \frac{\mu \eta_i}{2} \)} = \frac{(j + B - 1) (j + B)}{(k + B - 1) (k + B)},  
    \end{align*}
    and consequently, 
    \begin{align*}
        t_{k+1}  \eta_j \prod_{i=j+1}^k \(1-\frac{\mu \eta_i}{2}\) & = \frac{(k + B)   (j + B - 1) (j + B)}{4 (j + B) (k + B - 1) (k + B)} \le \frac14, 
    \end{align*}
    which shows that $\eta_k \in \(0, \frac1L\]$ and $t_{k+1} \eta_j \prod_{i=j+1}^k (1-\frac{\mu \eta_i}{2}) \in \(0, \frac14\]$, $\forall k \in \mathbb{N}$, $0 \leq j \leq k$. Therefore, we can apply Lemma \ref{lemma3} iteratively to get that}, for any $k \in \mathbb{N}$,
    \begin{align*}
        &\; \E \[ \exp\( t_{k+1}   \Delta_{k+1} \) \] \\
        \le&\; \E \[ \exp\( \(1 - \frac{\mu \eta_k}{2}\) t_{k+1}   \Delta_k + t_{k+1}   L \eta_k^2 \) \] \\
        \le&\; \E \[ \exp\( \(1 - \frac{\mu \eta_{k-1}}{2}\) \(1 - \frac{\mu \eta_k}{2}\) t_{k+1}   \Delta_{k-1} + \(1 - \frac{\mu \eta_k}{2}\) t_{k+1}   L \eta_{k-1}^2 + t_{k+1}   L \eta_k^2 \) \] \\
        \le&\; \cdots \\
        \le&\; \E \[ \exp\( t_{k+1}   \prod_{i=0}^k \( 1 - \frac{\mu \eta_i}{2} \) \Delta_0 + t_{k+1}  L \sum_{j=0}^k \eta_j^2 \prod_{i=j+1}^k \( 1 - \frac{\mu \eta_i}{2} \) \) \],  
    \end{align*}
    where 
    \begin{align*}
        t_{k+1} \prod_{i=0}^k \( 1 - \frac{\mu \eta_i}{2} \) \Delta_0  = \frac{\mu (B - 2) (B - 1) (k + B)  }{16 (k + B - 1) (k + B)} \Delta_0 \le \frac{\mu (B - 1)}{16} \Delta_0, 
    \end{align*}
    and
    \begin{align*}
        t_{k+1} L \sum_{j=0}^k \eta_j^2 \prod_{i=j+1}^k \( 1 - \frac{\mu \eta_i}{2} \) & = \frac{L (k+B)  }{\mu} \sum_{j=0}^k \frac{(j + B - 1) (j + B)}{(j + B)^2 (k + B - 1) (k + B)} \\
        & \le \frac{L (k + 1) }{\mu (k + B - 1)} \le \frac{L}{\mu}, 
    \end{align*}
    which implies 
    \begin{align*}
        \E \[ \exp\( t_{k+1} \Delta_{k+1} \) \] \le \exp \( \frac{\mu (B - 1)}{16} \Delta_0 + \frac{L}{\mu} \), \quad \forall k \in \mathbb{N}.   
    \end{align*}
    For $k=0$, we directly have 
    \begin{align*}
        \E \[ \exp \( t_0 \Delta_0 \) \] = \exp \( \frac{\mu (B - 1)}{16} \Delta_0 \) \le \exp \( \frac{\mu (B - 1)}{16} \Delta_0 + \frac{L}{\mu} \). 
    \end{align*}
    Thus, combining both cases, we obtain 
    {\if\highlight1\color{red}\fi
    \begin{align}
        \E \[ \exp\( t_k \Delta_k \) \] \le \exp \( \frac{\mu (B - 1)}{16} \Delta_0 + \frac{L}{\mu} \), \quad \forall k \in \mathbb{N}. 
        \label{eq:mgf-tdelta}
    \end{align}
    }

    By Markov's inequality, for any $v > 0$ and $k \in \mathbb{N}$, 
    \begin{align*}
        \Pr \( t_k \Delta_k {\if\highlight1\color{red}\fi >} \log \frac{1}{v} \) = \Pr \( \exp \( t_k \Delta_k \) {\if\highlight1\color{red}\fi >} \frac{1}{v} \) \leq  v \E \[ \exp \( t_k \Delta_k \) \] \le v \cdot \exp \( \frac{\mu (B - 1)}{16} \Delta_0 + \frac{L}{\mu} \). 
    \end{align*}
    We pick $v$ so that $\beta := v \cdot \exp \( \frac{\mu (B - 1)}{16} \Delta_0 + \frac{L}{\mu}\) \in (0, 1)$ to get 
    \begin{align*}
       \Pr \( f(x_k) - f^* \le \frac{16 \mu \log (1 / \beta) + \mu^2 (B - 1) \Delta_0 + 16 L}{\mu^2 (k + B - 1)} \) \ge 1 - \beta, 
    \end{align*}
    for any $k \in \mathbb{N}$ and $ \beta \in (0, 1)$. 
\end{proof}

We are now prepared to establish the main upper bound, presented in Theorem \ref{thm:upper}.

\begin{theorem}
    \label{thm:upper}
    Consider the classic stochastic gradient descent algorithm {\if\highlight1\color{red}\fi\eqref{eq:SGD}}, 
    where $ g (x_k ; \xi_k) $ is the stochastic gradient that satisfies Assumption \ref{assump:subGaussian}. 
    Let the objective $f$ be {\if\highlight1\color{red}\fi$\mu$-strongly convex} $(\mu > 0)$ and $L$-smooth $(0 < L < \infty)$. 
    For $k \in \mathbb{N}$, we set 
    \begin{align*}
        \eta_k = \frac{4}{\mu (k + B)}, 
        {\if\highlight1\color{red}\fi \quad t_{k} = \frac{\mu (k + B - 1)}{16},} 
    \end{align*}
    where $B := \max \left\{ \frac{4 L}{\mu}, 3 \right\}$. Then it holds that, for any $\beta \in \(0, \frac{6}{\pi^2} \)$, 
    \begin{align*} 
        \Pr \( \Delta_k \le \frac{16 \mu \log \( \( \log_2 k + 1 \)^2 / \beta \) + \mu^2 B \Delta_0 + 32 L }{ \mu^2 \( k + B - 1 \)}, \;  \forall k \in \mathbb{N}_+ \) \ge 1 - \frac{\pi^2 \beta }{6} , 
    \end{align*}
    where $\Delta_k := f (x_k) - f ^*$, and $ f^* $ is the minimum of $f$ over $\R^d$. 
\end{theorem}

\begin{proof}
{\if\highlight1\color{red}\fi
Since for any $k \in \mathbb{N}$, 
\begin{align*}
    \(1-\frac{\mu\eta_k}{2}\)t_{k+1} = \frac{k + B - 2}{k + B}\cdot \frac{\mu(k+B)}{16} = \frac{\mu(k + B - 2)}{16} \le t_k,  
\end{align*}
we can apply Lemma \ref{lemma4} to get that 
\begin{align*}
   \E \left[ \exp \( t_{k + 1} \Delta_{k + 1} \) \middle|\mathcal{F}_k \right] 
    \le \exp \( t_k \Delta_k + t_{k + 1} L \eta_k^2 \) 
    = \exp \( t_k \Delta_k + \frac{L}{\mu (k + B)} \). 
\end{align*}
}
Using the inequality $ \log(k+B) - \log(k-1+B)  \ge \frac{1}{k + B}$, we obtain
\begin{align*}
    \E \left[ \exp \( t_{k+1} \Delta_{k+1} - \frac{L}{\mu} \log(k + B) \) \middle| \mathcal{F}_k \right] \le \exp \( t_k \Delta_k - \frac{L}{\mu} \log(k - 1 + B) \), 
\end{align*}
thus $\left\{ \exp \left( t_k \Delta_k - \frac{L}{\mu} \log (k + B - 1) \right) \right\}_k$ is a positive supermartingale. We partition the iteration horizon into geometrically expanding intervals of the form $[2^m,2^{m+1})$, and consider 
\begin{align*}
    &\; \Pr \( \max_{ 2^m \le k < 2^{m+1} } t_k \Delta_k  > \log \frac{1}{v}  \) 
    =
    \Pr \( \max_{ 2^m \le k < 2^{m+1} } \exp \( t_k \Delta_k \)  > \frac{1}{v}  \) \\
    =& \;  
    \Pr \left\{ \max_{ 2^m \le k < 2^{m+1} } \exp \( t_k \Delta_k - \frac{L}{\mu} \log (2^{m+1} + B - 1) \) >  \frac{1}{v} \exp \( - \frac{L}{\mu} \log (2^{m+1} + B - 1) \)  \right\} \\
    \le& \; 
    \Pr \left\{ \max_{ 2^m \le k < 2^{m+1} }  \exp \(  t_k \Delta_k  - \frac{L}{\mu} \log (k + B - 1) \) > \frac{1}{v} \exp \( - \frac{L}{\mu} \log (2^{m+1} + B - 1) \)  \right\} =: \text{\textcircled{1}}. 
\end{align*} 
We then apply Ville's inequality to obtain, for any $v > 0$, $m \in \mathbb{N}$,
\begin{align*} 
    \text{\textcircled{1}}\overset{\text{(i)}}{\le}& \; 
    \frac{ v \E \[ \exp \( t_{2^m} \Delta_{2^m} - \frac{L}{\mu} \log (2^m + B - 1) \) \] }{ \exp \( - \frac{L}{\mu} \log (2^{m+1} + B - 1) \)  }  \\ 
    \overset{\text{(ii)}}{\le} & \; 
    2^{\frac{L}{\mu}} v  \E \[ \exp \( t_{2^m} \Delta_{2^m} \) \] \\
    \overset{\text{(iii)}}{\le} & \;
    v \exp \( \frac{\mu (B - 1)}{16} \Delta_0 +  \frac{\(1 + \log 2\)L}{\mu} \) \\
    \le & \; v \exp \( \frac{\mu B }{16} \Delta_0 +  \frac{2 L}{\mu} \) , 
\end{align*}
where (i) uses Ville's inequality, (ii) uses 
\begin{align*}
    \frac{\exp \( - \frac{L}{\mu} \log (2^m + B - 1) \)}{\exp \(- \frac{L}{\mu} \log (2^{m+1} + B - 1) \) } 
    = \( \frac{2^{m+1} + B - 1}{2^m + B - 1} \)^{\frac{L}{\mu}} < 2^{\frac{L}{\mu}}, 
\end{align*}
and (iii) uses {\if\highlight1\color{red}\fi\eqref{eq:mgf-tdelta}. }

We pick $v$ so that $\beta := v \exp \( \frac{\mu B }{16} \Delta_0 + \frac{2L}{\mu}  \) \in (0, 1)$ to get: For all $ \beta \in (0, 1) $ and all $ m \in \mathbb{N} $, 
\begin{align} 
\begin{aligned}
    &\; \Pr \( \Delta_k \le \frac{16 \mu \log (1/\beta) + \mu^2 B \Delta_0 + 32 L }{ \mu^2 \( k + B - 1 \)}, \;  2^m \le k < 2^{m + 1} \) \ge 1 - \beta. 
\end{aligned}
\label{eq:2^m}
\end{align}


Let 
\begin{align*}
    \mathcal{E}_m (\beta) 
    = 
    \left\{ \Delta_k \le \frac{16 \mu \log (1/\beta) + \mu^2 B \Delta_0 + 32 L }{ \mu^2 \( k + B - 1 \)}, \; \forall k, \;  2^m \le k < 2^{m + 1} \right\}. 
\end{align*}
For all $\beta \in \left( 0, \frac{6}{\pi^2} \right)$ and $m \in \mathbb{N}$, it holds that $\frac{\beta}{(m+1)^2} \le \beta < 1$, so by \eqref{eq:2^m} we have 
    $\Pr \( \mathcal{E}_m \( \frac{\beta}{(m+1)^2} \) \) \ge 1 - \frac{\beta}{(m+1)^2}. $
Note that $m  \le \log_2 k$ when $2^m \le k < 2^{m+1}$, by taking a union bound, we conclude that 
\begin{align*} 
    &\; \Pr \( \Delta_k \le \frac{16 \mu \log \( \( \log_2 k+1 \)^2 / \beta \) + \mu^2 B \Delta_0 + 32 L }{ \mu^2 \( k + B - 1 \)}, \quad  \forall k \in \mathbb{N}_+ \) \\ 
    \ge&\; 
    \Pr \( \bigcap_{m \in \mathbb{N} } \mathcal{E}_m \( \frac{\beta}{(m+1)^2} \) \) \\
    \ge & \;
    1 - \sum_{m=0}^{\infty} \frac{\beta}{(m+1)^2} = 1 - \frac{\pi^2 \beta }{6}, \quad \forall \beta \in \(0, \frac{6}{\pi^2}\) . 
\end{align*}

\end{proof}  

\section{The Uniform-in-Time Lower Bound}

In this section, we present a lower bound for the uniform-in-time convergence of stochastic gradient methods on strongly convex objectives. To begin with, we give a precise specification of the problem to be studied. We consider the stochastic first-order problem $\min_{x\in\R^d}f(x)$, in an environment where $f$ is accessible only through a stochastic gradient oracle $\Phi$. 
A stochastic first-order optimizer $\mathcal{M}$ is any procedure that solves these problems by iteratively selecting values from $\R^d$ and calling the oracle $\Phi$. 
More specifically, in the $k$-th iteration, the optimizer $\mathcal{M}$ calls the oracle at $x_k$, and the oracle reveals a stochastic version $g(x_k,\xi_k)$ of the gradient $\nabla f(x_k)$, where $\{\xi_k\}$ is an i.i.d. sequence of random variables following distribution $\Pr_{\Phi}$. The optimizer then uses the information $\{g(x_1,\xi_1),\cdots, g(x_k,\xi_k)\}$ to decide at which point $x_{k+1}$ the next oracle call should be made. A stochastic first-order oracle $\Phi$ is fully characterized by its gradient function $g(\cdot,\cdot)$ and the  noise distribution $\Pr_\Phi$. Throughout our analysis, we denote a specific problem instance by the pair $(f,\Phi)$, where $f$ is the objective function and $\Phi$ is the stochastic first-order oracle. Our focus in this section is on strongly convex problems, which are formalized as follows.
\begin{definition}
    \label{def}
    Given $ \mu, \sigma >0 $, the class $\Psi_{sc}(\mu, \sigma)$ consists of all problems $(f,\Phi)$ such that:
    \begin{enumerate}
        \item $f$ is $\mu$-strongly convex,
        \item For any $x$ and $k$, the stochastic gradient $g(x,\xi_k)$ satisfies $\E_{\Phi}[g(x,\xi_k)]=\nabla f(x)$ and 
        \begin{align*}
        \E_{\Phi}\left[\exp\left(\frac{\|g(x,\xi_k)-\nabla f(x)\|^2}{\sigma^2}\right)\right]\leq\exp(1).
        \end{align*}
    \end{enumerate}
\end{definition}

We would like to derive the lower bound for the following uniform-in-time convergence rate.
\begin{definition}
    Let $\Psi_{sc} (\mu, \sigma)$ be the set of stochastic first-order problems specified in Definition \ref{def} for some constants $\mu, \sigma > 0$. A stochastic gradient method $\mathcal{M}$ has uniform-in-time convergence rate $\varepsilon(\cdot):\;\mathbb{N}_+\to\R_+$ on $\Psi_{sc} (\mu, \sigma)$ with confidence $1-\alpha$ if for any $(f,\Phi)\in \Psi_{sc}(\mu, \sigma) $, 
    it holds that
    \begin{align*}
        \Pr_\Phi\left(f(x_k )-f^*\leq\varepsilon(k ), \;\; \forall k \in \mathbb{N}_+ \right)\geq1-\alpha.
    \end{align*}
\end{definition}

The following theorem provides a lower bound of the uniform-in-time convergence rate for stochastic gradient methods on strongly convex problems.

\begin{theorem}\label{thm:lower}
    Let $\alpha\in(0, 1)$, $\mu>0$ and $\sigma>0$ be constants. If a stochastic gradient method $\mathcal{M}$ has uniform-in-time convergence rate $\varepsilon(\cdot)$ on $\Psi_{sc}(\mu, \sigma)$ with confidence $1-\alpha$, then for infinitely many $n\in\mathbb{N}_+$, it holds that
        $\varepsilon(n)\geq\frac{\sigma^2(1-\alpha)}{48\mu}\cdot\frac{\log\log n}{n}.$
\end{theorem}

To prove this theorem, we will show that if we possess an optimizer with a uniform-in-time convergence guarantee, we can then construct a sequence of tests to progressively discriminate between increasingly closer distributions. We denote $\mathcal{V}=\{0,1\}^\infty$, the set of all $\{0,1\}$-valued sequences with infinite length. For each $v\in\mathcal{V}$, we define
\begin{align*}
    \theta(v)=2\cdot\sum_{i=1}^\infty v_i3^{-i}\quad\text{and}\quad \Pr_v=\mathcal{N}\left(\mu\theta(v),\sigma^2\right).
\end{align*}
This definition yields that if $v_{1:k}=v'_{1:k}$, then the KL-divergence between $\Pr_v$ and $\Pr_{v'}$ satisfies
\begin{align}
\begin{aligned}
    D_{kl}(\Pr_v\|\Pr_{v'})&=\frac{\mu^2\left(\theta(v)-\theta(v')\right)^2}{2\sigma^2} \leq\frac{\mu^2}{2\sigma^2}\left(2\cdot\sum_{i=k+1}^\infty3^{-i}\right)^2\leq\frac{9^{-k}\mu^2}{2\sigma^2}.
\end{aligned}
\label{eq:kl}
\end{align}
We will adopt the following concept of test sequence. 
\begin{definition}[\citet{duchi2024information}]
    {\if\highlight1\color{red}\fi
    Let $\mathcal{X}$ be a set and let $X_1,X_2,\cdots\overset{\text{i.i.d.}}{\sim}\Pr_v$ be $\mathcal{X}$-valued random variables. 
    For $\{n_k\}_{k \in \mathbb{N}}$ that is a sub-sequence of natural numbers,
    consider a sequence of  functions $ \left\{\widetilde{V}_k\right\}_{k\in\mathbb{N}_+} $ such that each $ \widetilde{V}_k : \mathcal{X}^{n_k} \to \{0,1 \} $ takes $ X_{1:n_k} := \( X_1, X_2, \cdots, X_{n_k} \) $ as input. 
    For $\alpha \in (0,1)$ and $\{n_k\}_{k \in \mathbb{N}}$ a subsequence of the natural numbers, 
    a sequence $\left \{\widetilde{V}_k\right\}_{k\in\mathbb{N}_+}$ defined with respect to $\{n_k\}_{k \in \mathbb{N}}$ is called an $(\alpha,\{n_k\})$-test (on $\{\Pr_v\}_{v \in \mathcal{V}}$) if for any $v\in\mathcal{V}$, 
    \begin{align*}
    \Pr_v\left(\widetilde{V}_k(X_{1:n_k})=v_k\text{ for all $k\in\mathbb{N}_+$}\right)\geq1-\alpha, 
    \end{align*}
    where $ v_k $ denote the $k$-th component of the sequence $ v \in  \mathcal{V}$. 
    }
\end{definition}
The following lemma demonstrates that a stochastic gradient method with a uniform-in-time convergence rate can be employed to construct a test sequence to distinguish normal distributions with increasingly closer locations.
\begin{lemma}\label{lem:reduction}
    Suppose $\mathcal{M}$ is a stochastic gradient method that has uniform-in-time convergence rate $\varepsilon(\cdot)$ on $\Psi_{sc}(\mu,\sigma)$ with confidence $1-\alpha$. Then for $n_k=\inf\left\{n:\varepsilon(n)<\frac{\mu}{8}\cdot9^{-k}\right\}$, there exists an $(\alpha,\{n_k\})$-test on $\{\Pr_v\}_{v\in\mathcal{V}}$.
\end{lemma}
\begin{proof}
    Suppose that $X_1,X_2,\cdots$ is an i.i.d. sequence following distribution $\mathcal{N}\left(\mu\theta(v),\sigma^2\right)$, where the location parameter $\theta(v)$ is unknown. In this proof, we will construct an $(\alpha,\{n_k\})$-test for each entry of $v$ based on optimizer $\mathcal{M}$.

    To begin with, we define a stochastic first-order problem based on the sequence $\{X_k\}_{k\in\mathbb{N}_+}$. For each $k$, we define the output of the gradient oracle as $g(x,X_k)=\mu x-X_k$. The objective function is defined as $f_{\theta(v)}$, where $f_{\theta(v)}(x)=\frac{\mu}{2}(x-\theta(v))^2$ for any $x\in\R$. Since $\E[g(x,X_k)]=\mu x-\mu\theta(v)$ for any $x$, it is easy to verify that $(f_{\theta(v)},\Phi)\in\Psi_{sc}(\mu,\sigma)$.
    
    Applying optimizer $\mathcal{M}$ on the problem $(f_{\theta(v)},\Phi)$ yields a sequence $\{x_k\}$. For each $\theta$, we define 
    \begin{align*}
        P_{\mathcal{V}}(\theta)=\mathop{\arg\min}_{v\in\mathcal{V}}|\theta-\theta(v)|
    \end{align*}
    and let $\widetilde{V}(X_{1:n_k})=P_\mathcal{V}(x_{n_k})_k$.
    Consider the event
    \begin{align*}
        E=\bigcap_{k\in\mathbb{N}_+}\{f_{\theta(v)}(x_k)-f_{\theta(v)}^*\leq\varepsilon(k)\}.
    \end{align*}
    Since $\mathcal{M}$ has uniform-in-time convergence rate $\varepsilon(\cdot)$ on $\Psi_{sc}(\mu,\sigma)$ with confidence $1-\alpha$, 
    we have 
    $\Pr(E)\geq1-\alpha$. On the event $E$, for each $k\in\mathbb{N}_+$, by strong convexity we have
    \begin{align*}
        |x_k-\theta(v)|\leq\sqrt{\frac{2}{\mu}(f_{\theta(v)}(x_k)-f_{\theta(v)}^*)}\leq\sqrt{\frac{2\varepsilon(k)}{\mu}},
    \end{align*}
    and thus
    \begin{align*}
        &\; |\theta(P_\mathcal{V}(x_{n_k}))-\theta(v)|\leq|\theta(P_\mathcal{V}(x_{n_k}))-x_{n_k}|+|x_{n_k}-\theta(v)| \leq 2\sqrt{\frac{2\varepsilon(n_k)}{\mu}}<3^{-k}.
    \end{align*}
    On the other hand, when $v$ and $v'$ satisfies $v_{1:k-1}=v'_{1:k-1}$ and $v_k\neq v'_k$, calculation yields that
    \begin{align*}
        \left|\theta(v)-\theta(v')\right|\geq2\cdot3^{-k}-2\cdot\sum_{i=k+1}^\infty3^{-i}=3^{-k}.
    \end{align*}
    
    Consequently, on event $E$ it satisfies $v_{1:k}=P_\mathcal{V}(x_{n_k})_{1:k}$ for each $k\in\mathbb{N}_+$, and thus $\left\{\widetilde{V}_k\right\}_{k\in\mathbb{N}_+}$ is an $(\alpha,\{n_k\})$-test.
\end{proof}

Now we are ready to give a proof of Theorem \ref{thm:lower}.

\begin{proof}[Proof of Theorem \ref{thm:lower}]
    Lemma \ref{lem:reduction} reduces the optimization problem to a test problem. 
    The proof is then completed by applying results for testing problems, relying on the following result of Duchi and Haque: 
    \begin{lemma}[Theorem 3 in \cite{duchi2024information}]\label{lem:nk}
        Let $\{\Pr_v\}_{v\in\mathcal{V}}$ be a collection of distributions where $D_{kl}(\Pr_v \| \Pr_{v'}) \leq \Delta_k$ holds for any $v, v' \in \mathcal{V}$ such that $v_{1:k} = v'_{1:k}$.
        If there exists an $(\alpha,\{n_k\})$-test, then for infinitely many $k\in\mathbb{N}_+$, $n_k>(1-\alpha)\Delta_k^{-1}\log k$. 
    \end{lemma}
    
    The above lemma, together with \eqref{eq:kl}, gives that for infinitely many $k\in\mathbb{N}_+$,
    \begin{align}\label{eq:lower-1}
        \varepsilon\left((1-\alpha)\frac{9^k\sigma^2}{2\mu^2}\log k\right)\geq\frac{\mu}{8}\cdot9^{-k}.
    \end{align}
    We denote $N_k=(1-\alpha)\frac{9^k\sigma^2}{2\mu^2}\log k$. Then from (\ref{eq:lower-1}) we have $9^k=\frac{2\mu^2N_k}{(1-\alpha)\sigma^2\log k}$ and
    \begin{align*}
        \varepsilon(N_k)\geq\frac{\mu}{8}\cdot\frac{(1-\alpha)\sigma^2\log k}{2\mu^2N_k}.
    \end{align*}
    Moreover, there exists a constant $K_0$ such that for any $k\geq K_0$, it holds that $N_k\leq e^{3k}$ and $\log k\geq\log\left(\frac{1}{3}\log N_k\right)\geq\frac{1}{3}\log\log N_k$. Consequently, for $n\in\{N_k\}_{k\geq K_0}$, we have
    \begin{align*}
        \varepsilon(n)\geq\frac{(1-\alpha)\sigma^2}{48\mu}\cdot\frac{\log\log n}{n},
    \end{align*}
    which concludes the proof.
\end{proof}

{\if\highlight1\color{red}\fi
\section{Conclusion}

In this work, we provide a high-probability analysis of uniform-in-time convergence bounds for standard stochastic gradient descent (SGD) applied to strongly convex objectives. 
We established improved upper bounds and provided matching lower bounds for this problem, demonstrating the tightness of a $\frac{\log\log k}{k}$ high-probability bound. Additionally, we improved upon existing last-iterate convergence rates by eliminating a logarithmic factor in the iteration count. While our analysis focuses on strongly convex objectives, the results naturally extend to Polyak--\L ojasiewicz functions. These contributions not only advance the theoretical understanding of SGD but also have practical implications for applications requiring dynamic termination criteria, such as early stopping in machine learning. 
}

\vspace{8pt}

\noindent\textbf{Acknowledgement.} T. Wang is partially supported by National Science Foundation of China (No. 12501704), the Science and Technology Commission of Shanghai Municipality (Shanghai Sailing Program, No. 24YF2702500). 
T. Wang is also partially funded by Shanghai Institute for Mathematics and Interdisciplinary Sciences (SIMIS) under grant numbers SIMIS-ID-2024-CN and SIMISID-2024-WE. T. Wang also acknowledge a grant supported by ByteDance Inc. 

\appendix 

\if\highlight1\color{red}\fi

\section{Uniform-in-Time Analysis for Contractive SA} \label{appendix:a}

The Stochastic Approximation (SA) algorithm solves root-finding or optimization problems in stochastic environments. To find the fixed point of $ \overline{F} $, the SA algorithm generates iterates of the form
\begin{align}
    x_{k+1} = x_k + \alpha_k \big( F(x_k;\xi_k) - x_k \big),
    \label{eq:sa-iter}
\end{align}
where $F(x_k;\xi_k)$ approximates the deterministic map $\overline{F}$ with additive sub-Gaussian noise, and $\alpha_k$ denotes the stepsize. When $\overline{F} : \mathbb{R}^d \to \mathbb{R}^d$ is contractive (i.e., There exists $\gamma \in (0,1)$ such that $ \| \overline{F} (x) - \overline{F} (y) \| \le \gamma \| x - y \| $ holds for all $x,y \in \R^d$), and the stepsize satisfies $\alpha_k \asymp 1/k$, the sequence $\{x_k\}$ converges to the fixed point of $\overline{F}$. Below, we demonstrate that Lemma \ref{lemma3} extends naturally to the convergence analysis of contractive SA, thereby preserving the main conclusions for constractive SA.

\begin{assumption} \label{ass:sa}
    Consider the SA iteration specified in \eqref{eq:sa-iter}. Let $ F(x_k;\xi_k) $ be the approximation for $\overline{F} : \R^d \to \R^d$ such that 
        
    (1) There exists $\gamma \in (0,1)$ such that $ \| \overline{F} (x) - \overline{F} (y) \| \le \gamma \| x - y \| $ holds for all $x,y \in \R^d$; 
        
    (2) $\E \left[ F(x_k; \xi_k) \middle| \mathcal{F}_k \right] = \overline{F} (x_k)$; 
        
    (3) (Uniform conditional 1-sub-Gaussianity, Theorem 2.6 in \cite{wainwright2019high}) 
    \begin{align*}
        \E \left[ \exp \( t \| F (x_k; \xi_k) - \overline{F} (x_k) \|^2 \) \middle| \mathcal{F}_k \right] \le \frac{1}{\sqrt{1 - 2t}}, \quad \forall t \in \( 0, \frac12 \), \forall k \in \mathbb{N}.
    \end{align*}
\end{assumption}
Notice that the first itemized condition is contractiveness, and the second and the third itemized conditions are Assumption \ref{assump:subGaussian}. Let $\Delta_k = \|x_k - x^*\|^2$, where $x^*$ is the fixed point of $\overline{F}$. 

\begin{lemma} \label{lem:A1}
    Let Assumption \ref{ass:sa} holds. If the stepsize consequence $\{\alpha_k\}_k$ and the nonnegative constant $t$ satisfy $\alpha_k \in \( 0, \frac{1}{1-\gamma} \]$, $t\alpha_k\in \(0, \frac{1-\gamma}{4}\]$ and $t\alpha_k^2 \in \(0, \frac{1}{8}\]$ for all $k\in\mathbb{N}$, then the following holds, 
    \begin{align*}
        \E \[ \exp \( t\Delta_{k+1} \) \middle| \mathcal{F}_k \] \leq \exp \( t (1-(1-\gamma)\alpha_k) \Delta_k + 2t\alpha_k^2 \), \quad \forall k \in \mathbb{N}. 
    \end{align*}
\end{lemma}

\begin{proof} 
We notice that 
\begin{align*} 
        & \; \|x_{k+1}-x^*\|^2 \\ 
        = & \; \left\| (1-\alpha_k) (x_k - x^*) + \alpha_k \( F(x_k;\xi_k) - \overline{F}(x^*) \) \right\|^2 \\
        = & \; (1-\alpha_k)^2 \| x_k - x^* \|^2 + 2(1-\alpha_k)\alpha_k \< x_k - x^*, F(x_k; \xi_k) - \overline{F}(x^*) \> + \alpha_k^2 \left\| F(x_k; \xi_k) - \overline{F}(x^*) \right\|^2 \\ 
        = & \; (1-\alpha_k)^2 \| x_k - x^* \|^2 + 2(1-\alpha_k)\alpha_k \< x_k - x^*, \theta_k \> + 2(1-\alpha_k)\alpha_k \< x_k-x^*, \overline{F}(x_k)-\overline{F}(x^*)\>  \\
        & + \alpha_k^2 \|\theta_k\|^2 + \alpha_k^2 \left\| \overline{F}(x_k) - \overline{F}(x^*) \right\|^2 + 2\alpha_k^2 \< \overline{F}(x_k) - \overline{F}(x^*), \theta_k \> \\
        \leq & \; (1-\alpha_k + \alpha_k\gamma)^2 \| x_k - x^* \|^2 + 2(1-\alpha_k)\alpha_k \< \theta_k, x_k-x^* \> + \alpha_k^2\|\theta_k\|^2 + 2\alpha_k^2 \< \theta_k, \overline{F}(x_k) - \overline{F}(x^*) \>, 
\end{align*} 
where $\theta_k = F(x_k; \xi_k) - \overline{F}(x_k)$ and the last inequality uses Cauchy-Schwarz inequality and the contractiveness of $\overline{F}$. 

The conditional MGF of $\Delta_{k+1}$ satisfies 
\begin{align}
    & \; \E \[ \exp \( t \Delta_{k+1} \) \middle| \mathcal{F}_k\] \nonumber \\
    = & \; \exp\( t(1-\alpha_k+\alpha_k\gamma)^2 \Delta_k \) \nonumber \\
    & \cdot \underbrace{\E \[ t \( 2(1-\alpha_k)\alpha_k \< \theta_k, x_k - x^* \>  + 2\alpha_k^2 \< \theta_k, \overline{F}(x_k) - \overline{F}(x^*) \> + \alpha_k^2\|\theta_k\|^2 \) \middle| \mathcal{F}_k \]}_{:=\text{\textcircled{1}}}. 
    \label{eq:mgf-sa}
\end{align}
By 1-sub-Gaussianity, we have 
\begin{align}
    \text{\textcircled{1}} & \overset{\text{(i)}}{\leq} \(\E \[ 4t\alpha_k \( (1-\alpha_k) \< \theta_k, x_k - x^* \> + \alpha_k \< \theta_k, \overline{F}(x_k) - \overline{F}(x^*) \> \) \middle| \mathcal{F}_k \] \cdot  \E \[ 2t\alpha_k^2 \| \theta_k \|^2 \middle| \mathcal{F}_k \] \)^{\frac12} \nonumber \\
    & \overset{\text{(ii)}}{\leq} \( \exp \( 8t^2 \alpha_k^2 \left\| (1-\alpha_k)(x_k-x^*) + \alpha_k \( \overline{F}(x_k) - \overline{F}(x^*) \) \right\|^2 \)  \cdot  \frac{1}{\sqrt{1-4t\alpha_k^2}} \)^{\frac12} \nonumber \\
    & \overset{\text{(iii)}}{\leq} \exp \( 4t^2\alpha_k^2(1-\alpha_k+\alpha_k\gamma)^2 \Delta_k + 2t\alpha_k^2 \),  
    \label{eq:1-sa}
\end{align}
where (i) uses Cauchy-Schwarz inequality, (ii) uses 1-sub-Gaussianity and (iii) uses $t\alpha_k^2 \in \(0, \frac18\]$ and the contractiveness of $\overline{F}$. Note that 
\begin{align}
    & \; \(t+4t^2\alpha_k^2\)(1-\alpha_k+\alpha_k\gamma)^2 \nonumber \\
    \overset{\text{(i)}}{\leq} & \; ( t + (1-\gamma)t\alpha_k ) \(1 - 2(1-\gamma)\alpha_k + (1-\gamma)^2\alpha_k^2\) \nonumber \\
    = & \; t(1-(1-\gamma)\alpha_k) - t(1-\gamma)^2\alpha_k^2(1-(1-\gamma)\alpha_k) \nonumber \\
    \overset{\text{(ii)}}{\leq} & \; t(1-(1-\gamma)\alpha_k), 
    \label{eq:delta-k-scale}
\end{align}
where (i) uses $t\alpha_k \in \( 0, \frac{1-\gamma}{4} \]$, (ii) uses $\alpha_k \in \(0, \frac{1}{1-\gamma}\]$. 

By substituting \eqref{eq:delta-k-scale} and \eqref{eq:1-sa} into \eqref{eq:mgf-sa}, we get the desired result. 
\end{proof} 

Building on Lemma \ref{lem:A1} and following the same proof technique, we can prove an $\mathcal{O}(k^{-1}\log\log k)$ upper bound for the SA algorithm \eqref{eq:sa-iter} under Assumption \ref{ass:sa} by setting the step size $\alpha_k \asymp 1/k$ and the parameter $t_k \asymp k$. 

\color{black}

\bibliographystyle{plainnat} 
\bibliography{references} 



\end{document}